\documentclass[12pt,twoside]{article}
\usepackage{amsmath, amsthm, amssymb}
\newtheorem{definition}{Definition}
\newtheorem{theorem}{Theorem}
\newtheorem{lemma}{Lemma}

\newtheorem{example}{Example}
\newtheorem{corollary}{Corollary}
\makeatletter \oddsidemargin0.45in \evensidemargin \oddsidemargin
\begin{document}
\thispagestyle{empty} \setcounter{page}{1}



\begin{center}
{\Large\bf  A  $\phi$ - contraction Principle in Partial Metric Spaces with Self-distance Terms}

\vskip.20in

T. Abdeljawad$^{a}$, Y. Zaidan$^{b,c}$, N. Shahzad $^{d,}$ \footnote{Corresponding Author E-Mail
Address:  nshahzad@kau.edu.sa} 
\\[2mm]
{\footnotesize $^a$Department of
Mathematics, \c{C}ankaya University, 06530, Ankara, Turkey}\\
{\footnotesize $^b$Department of Mathematics and Physical
Sciences,  Prince Sultan University\\
{\footnotesize P. O. Box 66833, Riyadh 11586, Saudi Arabia}\\
{\footnotesize $^c$Department of Mathematics, University of Wisconsin--Fox Valley}\\
{\footnotesize Menasha, WI 54952, USA}\\
{\footnotesize $^d$ Department of  Mathematics, Faculty of Sciences, King Abdulaziz University}\\
{\footnotesize Jedda, Saudi Arabia}}\\

\end{center}

\vskip.2in

{\footnotesize \noindent {\bf Abstract.} We prove a generalized contraction principle with control function in complete partial metric spaces. The contractive type condition used allows the appearance  of self distance terms. The obtained result generalizes some previously  obtained results such as the very recent " D. Ili\'{c}, V. Pavlovi\'{c} and V. Rako\u{c}evi\'{c}, Some new
extensions of Banach's contraction principle to partial metric
spaces, Appl. Math. Lett. 24 (2011), 1326--1330". An example is given to illustrate the generalization and its properness. Our presented example does not verify the contractive type conditions of the main results  proved recently  by S. Romaguera in " Fixed point theorems for generalized contractions on partial metric spaces, Topology Appl. 159 (2012), 194-199" and by I. Altun, F. Sola and H. Simsek in "Generalized contractions on partial metric spaces,
Topology and Its Applications 157 (18) (2010),  2778--2785". Therefore, our results have an advantage over the previously obtained.
}

\vskip.1in

\section{Introduction and Preliminaries} \label{s:1}
Banach contraction mapping principle is considered to be the key  of many extended fixed point theorems. It has widespread  applications in many branches of mathematics, engineering and computer.
Perviously many authors were able  to generalize this principle ( \cite{Khan}, \cite{Rhoades}, \cite{Dutta},\cite{BW1969}). After the appearance of partial metric spaces as a place for distinct research work into flow analysis, non-symmetric topology and domain theory (\cite{Mat92}, \cite{Mat94}), many authors started to generalize this principle to these spaces (see \cite{TEKcommon}, \cite{TEKcontrol}, \cite{Thweakly}, \cite{0V2004},\cite{02005}, \cite{IATOP2010}, \cite{IAFTPA2011}, \cite{wasfi}, \cite{Romag}, \cite{ATR}, \cite{Altun 2011} ). However, the contraction type conditions used in those generalizations do not reflect the structure of partial metric space apparently. Later, the authors in \cite{Rako} proved a more reasonable contraction principle in partial metric space in which  they used self distance terms. In this article we present a $\phi-$contraction principle in partial metric spaces. The presented contractive condition allows the self-distance to appear so that completeness, rather than the 0-completeness, of the partial metric space is needed.

We recall some definitions of partial metric spaces and state some
of their properties. A partial metric space (PMS) is a pair $(X,
p:X\times X\rightarrow\mathbb R^+)$ (where $\mathbb R^+$ denotes
the set of all non negative real numbers) such that
\begin{itemize}
\item[]
\begin{itemize}
    \item [(P1)] $p(x,y) =p(y,x)$\;\;(symmetry);
    \item [(P2)] If $0\leq p(x,x) = p(x,y) = p(y,y)$ then $x=y$\;\;
(equality);
    \item [(P3)] $p(x,x) \leq p(x,y)$\;\;   (small
    self--distances);

    \item [(P4)] $p(x,z) + p(y,y) \leq p(x,y) + p(y,z)$\;\;
(triangularity);
\end{itemize}
\end{itemize}
for all $x,y,z \in X$.

For a partial metric $p$ on $X$,  the function $p^s : X \times X
\rightarrow \mathbb R^+$ given by
\begin{equation}
p^s(x,y)=2p(x,y)-p(x,x)-p(y,y) \label{elma1}
\end{equation}
is a (usual) metric on $X$. Each partial metric $p$ on $X$
generates a $T_0$ topology $\tau_p$ on $X$ with a base of the
family of open $p$-balls $\{B_p(x, \varepsilon) : x\in X,
\varepsilon> 0\}$, where $B_p(x, \varepsilon) =\{y \in X : p(x, y)
< p(x, x) +\varepsilon\}$ for all $x \in X$ and $\varepsilon >0 $.

\begin{definition}
\emph{\cite{Mat94}}
\begin{itemize}
\item[$(i)$] A sequence $\{x_n\}$ in a PMS $(X,p)$ converges to
$x\in X$ if and only if $p(x,x)= \lim_{n \rightarrow \infty}
p(x,x_n)$. \item[$(ii)$] A sequence $\{x_n\}$ in a PMS $(X,p)$ is
called Cauchy if and only if $\lim_{n,m\rightarrow \infty}
p(x_n, x_m)$ exists (and finite). \item[$(iii)$] A PMS $(X,p)$ is
said to be complete if every Cauchy sequence $\{x_n\}$ in $X$
converges, with respect to $\tau_p$, to a point $x\in X$ such that
$p(x, x)=\lim_{n,m_\rightarrow \infty} p(x_n, x_m)$. \item[$(iv)$]
A mapping $T:X \rightarrow X$ is said to be continuous at $x_0 \in
X$, if for every $ \varepsilon> 0$, there exists $\delta > 0$ such
that $T(B_p(x_0,\delta))\subset B_p(T(x_0),\varepsilon)$.
\end{itemize}
\end{definition}

\begin{lemma}\emph{\cite{Mat94}}\label{lem-Mat}
\begin{enumerate}
\item[$(a1)$] A sequence $\{x_n\}$ is Cauchy in a PMS $(X,p)$ if
and only if $\{x_n\}$ is Cauchy in a metric space $(X,p^s)$.
\item[$(a2)$] A  PMS $(X,p)$ is complete  if and only if the
metric space $(X,p^s)$ is complete. Moreover,
\begin{equation}
\lim_{n \rightarrow \infty}p^s(x,x_n)=0 \Leftrightarrow p(x,
x)=\lim_{n \rightarrow \infty}p(x,x_n)=\lim_{n,m_\rightarrow
\infty} p(x_n, x_m).
 \label{Mat}
\end{equation}

\end{enumerate}
\end{lemma}

A sequence $\{x_n\}$ is called
0-Cauchy \cite{Rako} if $\mbox{lim}_{m,n} p(x_n,x_m)=0$. The partial metric space $(X,p)$ is called 0-complete if every 0-Cauchy sequence in $x$ converges to a point $x\in X$ with respect to $p$ and $p(x,x)=0$. Clearly, every complete partial metric space is complete. The converse need not be true.

\begin{example} \label{neednot} (see \cite{Rako})
Let $X=\mathbb{Q}\cap [0,\infty)$ with the partial metric $p(x,y)=max\{x,y\}$. Then $(X,p)$ is a 0-complete partial metric space which is not complete.
\end{example}

\begin{definition} \label{min radius}
Let $(X,p)$ be a complete metric space. Set $\rho_p=\mbox {inf}\{p(x,y) : x,y \in X\}$ and define $X_p=\{x \in X : p(x,x)=\rho_p\}$.
\end{definition}

\begin{theorem} \emph{\cite{Rako}} \label{Rako}
Let $(X,p)$ be a complete metric space, $\alpha \in [0,1)$ and $T
: X\rightarrow X$ a given mapping. Suppose that for each $x, y \in
X$ the following condition holds $$p(x,y)\leq \mbox{max}\{\alpha
p(x,y), p(x,x), p(y,y)\}.$$ Then
\begin{itemize}
\item[(1)] the set $X_p$ is nonempty; \item[(2)] there is a unique
$u \in X_p$ such that $Tu=u$; \item[(3)] for each $x \in X_p$ the
sequence $\{T^nx\}_{n\geq1}$ converges with respect to the metric
$p^s$ to $u$.
\end{itemize}
\end{theorem}

The proof of the following lemma can be easily achieved by using
the partial metric topology.

\begin{lemma} \emph{\cite{TEKcommon,Thweakly}} \label{continuity of pm}
Assume $x_n\rightarrow z$ as $n\rightarrow\infty$ in a PMS $(X,p)$
such that $p(z,z)=0$. Then $\lim_{n\rightarrow\infty}
p(x_n,y)=p(z,y)$ for every $y \in X$.
\end{lemma}

The following Lemma summarizes the relation between certain comparison functions that usually act as control functions in the studied contractive typed mappings in fixed point theory. For such a summary and fixed point theory for $\phi-$ contractive mappings we advice for \cite{A Rus}.
\begin{lemma}
Let $\phi: \mathbb{R}^+\rightarrow \mathbb{R}^+$ be a function and relative to the function $\phi$ consider the following conditions:

\begin{itemize}
  \item (i) $\phi$ is monotone increasing.
  \item (ii) $\phi(t)< t$ for all $t>0$.
  \item (iii) $\phi(0)=0$.
  \item (iv) $\phi$ is right uppersemicontinuous.
  \item (v) $\phi$ is right continuous.
  \item (vi) $\lim_{n\rightarrow \infty} \phi^n(t)=0$ for all $t\geq 0$.
\end{itemize}
Then the following are valid:
\begin{itemize}
  \item (1) The conditions (i) and (ii) imply (iii).
  \item (2) The conditions (ii) and (v) imply (iii).
  \item (3) The conditions (i) and (vi) imply (ii).
  \item (4) The conditions (i) and (iv) imply imply (vi).
  \item (5) If $\phi$ satisfies (i) then (iv) $\Leftrightarrow$ (v).
\end{itemize}

\end{lemma}
\section{Main Results } \label{s:1}

\begin{theorem} \label{mainthm}
Let $(X,p)$ be a complete partial metric space. Suppose
$T: X \rightarrow X$ is a given mapping satisfying:
\begin{eqnarray}
p(Tx,Ty)\leq \max\{\phi(p(x,y)),p(x,x),p(y,y)\},\label{phicontr}
\end{eqnarray}
where $\phi: [0,\infty)\rightarrow [0,\infty)$ is an increasing function
 such that $f(t)=t-\phi(t)$ is increasing with $f^{-1}$ is right continuous at $0$ . Also assume $\lim_{n\rightarrow \infty}\phi^n(t)=0$ for all $t\geq0$ (and hence $\phi(0)=0, \phi(t)<t$ for $t>0$ ). Then:
\begin{itemize}
\item[(1)] the set $X_p$ is nonempty;
\item[(2)] there is a unique $u \in X_p$ such that $Tu=u$;
\item[(3)] for each $x \in X_p$ the sequence $\{T^nx\}_{n\geq1}$ converges with respect to the metric
$p^s$ to $u$.
\end{itemize}
\end{theorem}
\begin{proof}
Let $x\in X$. Then $p(Tx,Tx)\leq p(x,x)$ and therefore $\{p(T^nx,T^nx)\}_{n\geq0}$ is a nonincreasing sequence.
Now Define $$M_x:=f^{-1}(p(x,Tx))+p(x,x),$$ where $f(t)=t-\phi(t)$. Notice that
$f(0)=0$ (and hence  $f^{-1}(0)=0$), $f(t)<t$ for $t>0$ and hence $f^{-1}(t)>t$ for $t>0$. Now we prove that
\begin{equation}\label{Mx}
p(T^nx,x)\leq M_x,\,\,\, \forall n\geq0.
\end{equation}
Inequality (\ref{Mx}) is true for $n=0, 1$ since: $p(x,x)\leq M_x$ and $p(Tx,x)\leq f^{-1}(p(Tx,x))\leq M_x$. Now we proceed by induction. Suppose that (\ref{Mx}) is true for each $n\leq n_0-1$ for some positive integer $n_0\geq2$. Then
\begin{eqnarray*}
p(T^{n_0}x,x)&\leq& p(T^{n_0}x,Tx)+p(Tx,x)\\
&\leq&\max\{\phi(p(T^{n_0-1}x,x)),p(T^{n_0-1}x,T^{n_0-1}x),p(x,x)\}+p(Tx,x)\\
&\leq&\max\{\phi(p(T^{n_0-1}x,x))),p(x,x)\}+p(Tx,x)\\
\end{eqnarray*}
\begin{itemize}
\item[]
\begin{itemize}
\item[Case 1:]
\begin{eqnarray*}
p(T^{n_0}x,x)&\leq& \phi(p(T^{n_0-1}x,Tx))+p(Tx,x)\\
&\leq& \phi(f^{-1}(p(Tx,x))+p(x,x))+p(Tx,x)\\
&=& f^{-1}(p(Tx,x))+p(x,x)-f(f^{-1}(p(Tx,x))+p(x,x))+p(Tx,x)\\
&\leq& M_x-f(f^{-1}(p(Tx,x)))+p(Tx,x)=M_x.
\end{eqnarray*}
\item[Case 2:]
\begin{eqnarray*}
p(T^{n_0}x,x)&\leq& p(x,x)+p(Tx,x)\\
&\leq&p(x,x)+f^{-1}(p(Tx,x))=M_x.
\end{eqnarray*}
\end{itemize}
\end{itemize}
Thus, we obtain (\ref{Mx}). Next we prove that the sequence $\{p(T^nx,T^nx)\}_{n\geq0}$ is Cauchy. Equivalently,
we show that
\begin{eqnarray}\label{cauchy}
\lim_{n,m\rightarrow\infty}p(T^nx,T^mx)=r_x
\end{eqnarray}
where $r_x:=\inf_{n}p(T^nx,T^nx)$. Now clearly
$r_x\leq p(T^nx,T^nx)\leq p(T^nx,T^mx)$ for all $n,m$. Also, given any $\epsilon>0$, there exists $n_0\in\mathbb{N}$ such that
$p(T^{n_0}x,T^{n_0}x)<r_x+\epsilon$ and $\phi^{n_0}(2M_x)<r_x+\epsilon$. Therefore, for any $m,n>2n_0$ we have
\begin{eqnarray*}
r_x&\leq& p(T^nx,T^mx)\\
&\leq&\max\{\phi(p(T^{n-1}x,T^{m-1}x)),p(T^{n-1}x,T^{n-1}x),p(T^{m-1}x,T^{m-1}x)\}\\
&\leq&\max\{\phi^2(p(T^{n-2}x,T^{m-2}x)),p(T^{n-2}x,T^{n-2}x),p(T^{m-2}x,T^{m-2}x)\}\\
&\leq&\max\{\phi^{n_0}(p(T^{n-n_0}x,T^{m-n_0}x)),p(T^{n-n_0}x,T^{n-n_0}x),p(T^{m-n_0}x,T^{m-n_0}x)\}\\
&\leq&\max\{\phi^{n_0}(p(T^{n-n_0}x,x)+p(T^{m-n_0}x,x)),p(T^{n-n_0}x,T^{n-n_0}x),p(T^{m-n_0}x,T^{m-n_0}x)\}\\
&<&\max\{\phi^{n_0}(2M_x),r_x+\epsilon,r_x+\epsilon\}\\
&<&r_x+\epsilon.
\end{eqnarray*}
Hence, we obtain (\ref{cauchy}). Since $(X,p)$ is a complete partial metric space, there exists $z\in X$ such that
$p(z,z)=r_x$. Next, we show that $p(z,z)=p(Tz,z)$. For every $n\in\mathbb{N}$ we have
\begin{eqnarray*}
p(z,z)&\leq& p(Tz,z)\\
&\leq& p(Tz,T^nx)+p(T^nx,z)-p(T^nx,T^nx)\\
&\leq&\max\{\phi(p(z,T^{n-1}x)),p(T^{n-1}x,T^{n-1}x),p(z,z)\}+p(T^nx,z)-p(T^nx,T^nx).
\end{eqnarray*}
\begin{itemize}
\item[]
\begin{itemize}
\item[Case 1:]
\begin{eqnarray*}
p(Tz,z)&\leq&\phi(p(z,T^{n-1}x))+p(T^nx,z)-p(T^nx,T^nx)\\
&\leq&p(z,T^{n-1}x)+p(T^nx,z)-p(T^nx,T^nx)\rightarrow p(z,z) \mbox{  as  } n\rightarrow\infty
\end{eqnarray*}
\item[Case 2:] $p(Tz,z)\leq p(T^{n-1}x,T^{n-1}x)+p(T^nx,z)-p(T^nx,T^nx)\rightarrow p(z,z) \mbox{  as  } n\rightarrow\infty$
\item[Case 3:] $p(Tz,z)\leq p(z,z)+p(T^nx,z)-p(T^nx,T^nx)\rightarrow p(z,z) \mbox{  as  } n\rightarrow\infty.$
\end{itemize}
\end{itemize}
Therefore,
\begin{equation}\label{eq}
p(z,z)=p(Tz,z)
\end{equation}
Now we show that $X_p$ (see Definition \ref{min radius}) is nonempty. For each $k\in \mathbb{N}$ choose $x_k\in X$ with
$p(x_k,x_k)<\rho_p+1/k$, where $x_k=T^kx$. First, we prove that
\begin{equation}\label{limz}
\lim_{m,n\rightarrow\infty} p(z_n,z_m)=\rho_p.
\end{equation}
Given $\epsilon>0$, take $n_0:=[f^{-1}(3/\epsilon)]+1$. If $k>n_0$, then
\begin{eqnarray*}
\rho_p&\leq& p(Tz_k,Tz_k)\leq p(z_k,z_k)=r_{x_k}\leq p(x_k,x_k)<\rho_p+1/k\\
&<&\rho_p+1/n_0<\rho_p+1/f^{-1}(3/\epsilon).
\end{eqnarray*}
Set $U_k:=p(z_k,z_k)-p(Tz_k,Tz_k)$. Then $U_k<1/f^{-1}(3/\epsilon)$ for $k>n_0$.
Thus, if $m,n>n_0$ then by (\ref{eq}) and the fact that $f$ (and hence $f^{-1}$) is increasing, we have
\begin{eqnarray*}
p(z_n,z_m)&\leq&p(z_n,Tz_n)+p(Tz_n,Tz_m)+p(Tz_m,z_m)-p(Tz_n,Tz_n)-p(Tz_m,Tz_m)\\
&=&U_n+U_m+p(Tz_n,Tz_m)\\
&<&2/f^{-1}(3/\epsilon)+\max\{\phi(p(z_n,z_m)),p(z_n,z_n),p(z_m,z_m)\}\\
&\leq&\max\{f^{-1}\left(2/f^{-1}(3/\epsilon)\right),3/f^{-1}(3/\epsilon)+\rho_p\}\\
&\leq&\max\{f^{-1}\left(2\epsilon/3)\right),\rho_p+\epsilon\}\\
&\leq&\rho_p+\epsilon+f^{-1}(2\epsilon/3).
\end{eqnarray*}
Therefore, if we let $\epsilon\rightarrow 0$ we get (\ref{limz}). Since $(X,p)$ is a complete partial metric space, there exists $u\in X$
such that $p(u,u)=\lim_{m,n\rightarrow\infty} p(z_n,z_m)=\rho_p$. Consequently, $u\in X_p$ and hence $X_p$ is nonempty.

Now choose an arbitrary $x\in X_p$. Then
$$\rho_p\leq p(Tz,Tz)\leq p(Tz,z)=p(z,z)=r_x=\rho_p,$$
which, using P2, implies that $Tz=z$. To prove uniqueness of the fixed point we suppose that
$u,v\in X_p$ are both fixed points of T. Then
\begin{eqnarray*}
\rho_p\leq p(u,v)=p(Tu,Tv)&\leq&\max\{\phi(p(u,v)),p(u,u),p(v,v)\}\\
&\leq&\max\{\phi(p(u,v)),\rho_p\}.
\end{eqnarray*}
\begin{itemize}
\item[]
\begin{itemize}
\item[Case 1:] $\rho_p\leq p(u,v)\leq\rho_p\Rightarrow p(u,v)=\rho_p=p(u,u)=p(v,v)\Rightarrow u=v$.
\item[Case 2:]
\begin{eqnarray*}
&&p(u,v)\leq \phi(p(u,v))\\
\Rightarrow &&p(u,v)-\phi(p(u,v))\leq0\\
\Rightarrow &&f(p(u,v))\leq0\\
\Rightarrow &&f(p(u,v))=0\\
\Rightarrow &&p(u,v)=0\\
\Rightarrow&&u=v
\end{eqnarray*}
\end{itemize}
\end{itemize}
Thus, the fixed point is unique.
\end{proof}
Clearly, the above theorem does not guarantee uniqueness of the fixed point in $X$. However, if (\ref{phicontr}) is
replaced by the condition below, we can show uniqueness.
\begin{theorem}\label{mainthm2}
Let $(X,p)$ be a complete partial metric space. Suppose
$T: X \rightarrow X$ is a given mapping satisfying:
\begin{eqnarray}
p(Tx,Ty)\leq \max\left\{\phi(p(x,y)),\frac{p(x,x)+p(y,y)}{2}\right\},\label{sphicontr}
\end{eqnarray}
where $\phi: [0,\infty)\rightarrow [0,\infty)$ is as in Theorem \ref{mainthm}. Then there is a unique point
$z\in X$ such that $Tz=z$. Furthermore, $z\in X_p$ and for each $x\in X_p$ the sequence $\{T^nx\}_{n\geq1}$ converges
with respect to the metric $p^s$ to $z$.
\end{theorem}
\begin{proof}
Using Theorem \ref{mainthm} we only need to prove uniqueness. Suppose there exists $u,v\in X$ such that
$Tu=u$ and $Tv=v$. Now $$p(u,v)=p(Tu,Tv)\leq\max\left\{\phi(p(u,v)),\frac{p(u,u)+p(v,v)}{2}\right\}.$$
\begin{itemize}
\item[]
\begin{itemize}
\item[Case 1:]
\begin{eqnarray*}
&&p(u,v)\leq \phi(p(u,v))\\
\Rightarrow &&p(u,v)-\phi(p(u,v))\leq0\\
\Rightarrow &&f(p(u,v))\leq0\\
\Rightarrow &&f(p(u,v))=0\\
\Rightarrow &&p(u,v)=0\\
\Rightarrow&&u=v
\end{eqnarray*}
\item[Case 2:]
\begin{eqnarray*}
&&p(u,v)\leq\frac{p(u,u)+p(v,v)}{2}\\
\Rightarrow &&2p(u,v)-p(u,u)-p(v,v)\leq0\\
\Rightarrow &&p^s(u,v)=0\\
\Rightarrow&&u=v
\end{eqnarray*}
\end{itemize}
\end{itemize}
\end{proof}

\begin{corollary}
Let $(X,p)$ be a 0-complete partial metric space. Suppose
$T: X \rightarrow X$ is a given mapping satisfying:
\begin{eqnarray}
p(Tx,Ty)\leq \phi(p(x,y)),\label{0-phicontr}
\end{eqnarray}
where $\phi: [0,\infty)\rightarrow [0,\infty)$ is an increasing function
 such that $f(t)=t-\phi(t)$ is increasing with $f^{-1}$ is right continuous at $0$ . Also assume $\lim_{n\rightarrow \infty}\phi^n(t)=0$ for all $t\geq0$ (and hence $\phi(0)=0, \phi(t)<t$ for $t>0$ ). Then there is a unique $z \in X$ such that $Tz=z$. Also $p(z,z)=0$ and for each $x \in X$ the sequence $\{T^nx\}$ converges with respect to the metric $p^s$  to $z$.
\end{corollary}

\begin{corollary}
If in Theorem \ref{mainthm} and Theorem \ref{mainthm2} the function $\phi(t)=\alpha t$, $\alpha \in (0,1]$, then Theorem \ref{Rako} and Theorem 3.2 in \cite{Rako} will follow.
\end{corollary}
\begin{example}
Let $X=[0,1]\cup [3,4]$. Define $p:X\times X\rightarrow\mathbb{R},\, T:X\rightarrow X$ and
$\phi: [0,\infty)\rightarrow [0,\infty)$ as follows:
\begin{eqnarray*}
&&p(x,y)=\max\{x,y\}\\
&&T(x)=\left\{\begin{array}{cll}
\frac{x}{2}&,& x\in [0,1]\\[1.5ex]
\frac{7}{5}&, & x\in [3,4]
\end{array}\right.\\
&&\phi(t)=\frac{t}{1+t}
\end{eqnarray*}
The above definitions satisfy the hypothesis of Theorem \ref{mainthm2}. In particular, we make the following observations:
\begin{itemize}
\item $(X,p)$ is a complete partial metric space.
\item We can easily prove by induction that $\phi^n(t)=\frac{t}{1+nt}$ which implies that $\lim_{n\rightarrow \infty}\phi^n(t)=0$.
\item $T$ satisfies condition (\ref{sphicontr}):
\begin{itemize}
\item[1)] If $\{x,y\}\cap[3,4]\neq\emptyset$ then
\begin{eqnarray*}
p(Tx,Ty)&=&\max\{Tx,Ty\}=\frac{7}{5}\\
&\leq& \max\left\{\phi(p(x,y)),\frac{p(x,x)+p(y,y)}{2}\right\}
\end{eqnarray*}
\item[2)] If $\{x,y\}\subset[0,1]$ then
\begin{eqnarray*}
p(Tx,Ty)&=&\max\{Tx,Ty\}=\max\left\{\frac{x}{2},\frac{y}{2}\right\}\\
&\leq& \max\left\{\phi(p(x,y)),\frac{p(x,x)+p(y,y)}{2}\right\}.
\end{eqnarray*}
\end{itemize}
\end{itemize}

\begin{itemize}
  \item By Theorem \ref{mainthm2} there is a unique fixed point which is $z=0$.
  \item On the other hand, if the partial metric  $p$ is replaced by the usual absolute value metric then it can be easily checked that condition (\ref{sphicontr}) is not satisfied with, for example, $x=1$ and $y=3$.
  \item we remark that this our example does not verify the conditions of the main theorem in \cite{IATOP2010}. Therefore, our result has  a benefit over  \cite{IATOP2010}.

   \item  Our example does not verify the conditions of Theorem 4 in \cite{Romag}. For example, the $\phi$-contractive condition appeared there is not satisfied for $x=3,~y=4$. Thus, it has an advantage over  \cite{Romag}.

   \item  Our example does not verify the conditions of Theorem 3 in \cite{Romag}. Check for  $x=3,~y=4$.

\end{itemize}

\end{example}

\end{document}